\documentclass{article}
\usepackage{natbib}
\usepackage{algorithm}
\usepackage{algpseudocode}
\usepackage{lmodern}
\usepackage[T1]{fontenc}
\usepackage[utf8]{inputenc}
\usepackage{amsfonts}
\usepackage{amssymb,amsbsy}
\usepackage{amsmath}
\usepackage{amsfonts, amscd, amsthm}
\usepackage{color}
\usepackage{url}
\usepackage{verbatim}
\usepackage{setspace}
\usepackage[abs]{overpic}
\usepackage{graphicx}
\usepackage{subfigure}
\usepackage{bbm, bm}
\usepackage{accents}
\usepackage{fullpage}

\newtheorem{theorem}{Theorem}
\newtheorem*{example*}{Example}
\numberwithin{equation}{section}

\newtheorem{definition}{Definition}

\newtheorem{lemma}{Lemma}

\newtheorem{proposition}{Proposition}
\newtheorem{remark}{Remark}

\newtheorem{assumption}[theorem]{Assumption}

\newcommand{\mb}{\boldsymbol}
\renewcommand{\bm}{\boldsymbol}
\newcommand{\mc}{\mathcal}
\newcommand{\mf}{\mathfrak}

\newcommand{\mcb}[1]{\mb{\mc{#1}}}
\newcommand{\norm}[2]{\left\| #1 \right\|_{#2}}
\newcommand{\reals}{\mathbb{R}}
\newcommand{\<}{\langle}
\renewcommand{\>}{\rangle}
\newcommand{\innerprod}[2]{\left\< #1, #2 \right\>}
\newcommand{\set}[1]{\left\{ #1 \right\}}

\newlength{\dhatheight}

\newcommand{\diag}[1]{\mathrm{diag}\left( #1 \right)}

\renewcommand\t{{\ensuremath{\scriptscriptstyle{\top}}}}

\numberwithin{equation}{section}
\newcommand{\mr}[1]{\mathbb{#1}}

\begin{document}

\title{Using Negative Curvature in Solving Nonlinear Programs}
\author{Donald Goldfarb$^1$, Cun Mu$^1$, John Wright$^2$, Chaoxu Zhou$^1$ \\ \\
$^1$Department of Industrial Engineering and Operations Research, Columbia University\\
$^2$Department of Electrical Engineering, Columbia University
}
\maketitle

\begin{abstract}
Minimization methods that  search along a curvilinear path composed of a non-ascent negative curvature direction in addition to the direction of steepest descent, dating back to the late 1970s, 
have been an effective approach to finding a stationary point of a function at which its Hessian is positive semidefinite. For constrained nonlinear programs arising from recent applications, the primary goal is to find a stationary point that satisfies the second-order necessary optimality conditions. Motivated by this, we generalize the approach of using negative curvature directions from unconstrained optimization to nonlinear ones. We focus on equality constrained problems and prove that our proposed negative curvature method is guaranteed to converge to a stationary point satisfying second-order necessary conditions. A possible way to extend our proposed negative curvature method to general nonlinear programs is also briefly discussed.
\end{abstract}

\section{Introduction}

This paper is concerned with solving general smooth nonlinear optimization problems. Considering the difficulty of obtaining a global optimal solution, our goal is to develop an efficient method to locate a stationary point that satisfies the second-order necessary conditions. This goal is motivated by recent applications and developments, where stationary points satisfying second-order necessary conditions are of primary interest. Specifically, in many nonlinear programs arising from dictionary learning \cite{SQW152}, tensor decomposition \cite{GHJY15}, phase retrieval \cite{SQW16}, and semidefinite programming \cite{BM03},  stationary points satisfying the second-order necessary conditions are the targets of the application and can often be proven to be global optimizers.

Even though the scope of our concern is more general than unconstrained nonlinear programs, studying such problems helps us to better understand how to achieve our goal. To find a stationary point of an unconstrained function satisfying second-order necessary conditions, apparently, using gradient information alone in a deterministic method is not sufficient as such methods can be trapped by saddle points. Therefore, second-order information should be utilized in (hopefully) some computationally efficient manner. Methods, e.g. {\em Newton's method} and {\em trust-region methods}, that involve computing the inverse, Cholesky factorization or eigen-decomposition of the Hessian or some modifications of it, might become impractical when the number of variables grow to tens of thousands. As a remedy, methods that utilize negative curvature directions were proposed in \cite{MC77, MS79, GOLD80} and later further actively developed in \cite{grandinetti1984nonlinear, martinez1990algorithm, ferris1996nonmonotone, lucidi1998curvilinear, gould2000exploiting, conforti2001curvilinear, apostolopoulou2010curvilinear} based on inexpensive computations involving the Hessian. Specifically, these methods imitate the steepest descent method but also incorporate a negative curvature direction besides the negative of the gradient when computing the step to take at each iteration. Consequently, saddle points are usually avoided and these methods can be guaranteed to converge to {\em a stationary point satisfying the second-order necessary conditions} without requiring much additional computational cost.

Inspired by the unconstrained case, we {\em generalize the negative curvature approach to equality constrained nonlinear programs}. In order to accomplish this, one natural question is how to generalize the notions of the gradient and the Hessian when constraints are imposed. For unconstrained problems, we can regard the gradient and the Hessian as quantities that characterize local optimality. Specifically, according to the first and second-order necessary optimality conditions, a local minimizer must have a zero gradient and a positive semidefinite Hessian. For constrained nonlinear programs, it seems plausible that a generalized gradient and Hessian could be developed based on the first-order and second-order optimality conditions. Essentially, this is the approach we adopt here. Our generalized gradient and Hessian are closely related to the Riemannnian gradient and Hessian in the context of optimization on Riemannian manifolds \cite{petersen2006riemannian, AMS09, absil2013extrinsic,  absil2009all}. This explicit connection is discussed in the paper. Our approach of deriving generalized gradient and Hessian from the classical optimality point of view is quite well-motivated and self-contained, and moreover keeps the technical difficulty to a minimum level.

\paragraph{Organization.} The rest of the paper is organized as follows. In Section \ref{sec:prelim}, we review some basic results on equality constrained optimization. In Section \ref{sec:eq}, we present our negative curvature line search method for equality-constrained nonlinear programs in particular. The notions of gradient and Hessian are generalized based on local optimality conditions. Using these generalizations, we propose and study a negative curvature method for equality constrained problem. Last, we briefly discuss how to extend the negative curvature method proposed for the equality constrained problem to general nonlinear programming problems.

\section{Preliminaries}\label{sec:prelim}

In this section, we review several fundamental results regarding the equality-constrained problem:
\begin{flalign} \label{eqn:problem_to_solve}
\text{minimize} \quad f(\bm x) \qquad
\text{subject to} \quad  c_i(\bm x) = 0, \;\;\; i\in \set{1, 2, \ldots, m},
\end{flalign}
where the objective function $f:\reals^n \to \reals$ and constraint functions $c_i:\reals^n \to \reals$ are smooth so that it is possible to characterize local optimality conditions based on their derivatives. The results and notation in this section are quite classical \cite{nocedal2006numerical}.
Regarding the feasible set $\Omega:=\set{\bm x\in \reals^n \;\vert\; c_i(\bm x) = 0}$, we first impose a regularity condition which is known as LICQ in optimization literature.
\begin{definition}[LICQ]
We say that the linear independence constraint qualification (LICQ) holds at $\bm x \in \Omega$, if the set of constraint gradients $\set{\nabla c_1(\bm x), \nabla c_2{(\bm x)}, \ldots, \nabla c_m{(\bm x)}}$ are linearly independent, i.e. the matrix $\nabla \mb{c}({\bm x}) := [\nabla c_1(\bm x), \nabla c_2{(\bm x)}, \ldots, \nabla c_m{(\bm x)}] \in \reals^{n \times m}$ has full column rank.
\end{definition}

From a geometric point of view, the LICQ condition guarantees ${\Omega}$ to be a differentiable manifold of dimension $n-m$. Now we are ready to state the definitions of tangent and normal subspaces, which are particularly useful geometry concepts to characterize the variational properties of $f(\cdot)$.
\begin{definition}[Tangent and Normal Subspaces]
The normal subspace $\mc N_{\Omega}(\bm x)$ at a point $\bm x \in \Omega$ is defined as the subspace spanned by the set of constraint gradients $\set{\nabla c_1(\bm x), \ldots, \nabla c_m{(\bm x)}}$, i.e. $\mc N_{\Omega}(\bm x):= \mf R(\nabla\mb c({\bm x})) \subseteq \reals^n$ the range of the matrix $\nabla\mb c({\bm x})$. We denote by $\mc P_{\mc N_{\Omega}(\bm x)}$ projection onto the normal subspace $\mc N_{\Omega}(\bm x)$.

The tangent subspace  $\mc T_{\Omega}(\bm x)$ at a point $\bm x \in \Omega$ complements the normal subspace in $\reals^n$, i.e. $\mc T_{\Omega}(\bm x) = \mc N^\perp_{\Omega}(\bm x) = \mf R(\nabla \mb c({\bm x}))^\perp = \mf N(\nabla \mb c({\bm x})^\top)$, is the nullspace of the matrix $\nabla \mb c({\bm x})^\top$.
We denote by $\mc P_{\bm x}$ the projection onto the tangent subspace $\mc T_{\Omega}(\bm x)$.
\end{definition}
\begin{remark}
We may drop the subscripts $\Omega$ and $\bm x$, respectively, from $\mc N_{\Omega}(\bm x)$ and $\mc T_{\Omega}(\bm x)$ if they are clear from the context.
\end{remark}

  We are now fully prepared to state the first-order and second-order necessary optimality conditions for problem (\ref{eqn:problem_to_solve}) based on the Lagrangian function,
\begin{flalign}
\mc L(\bm x, \bm \lambda):= f(\bm x) - \sum_{i = 1}^m\; \lambda_i \bm c_i(\bm x),
\end{flalign}
where $\lambda_i$ is the Lagrange multiplier associated with the $i$-th constraint $c_i(\bm x) = 0$.

\begin{theorem}[Necessary Optimality Conditions \cite{nocedal2006numerical}]
Suppose that $\bm x^\star$ is a local minimizer of problem \eqref{eqn:problem_to_solve} and satisfies the LICQ condition.
Then $\mb{x}^\star$ satisfies the following first-order and second-order conditions
\begin{flalign}
& \mc{G}(\bm x^\star):=\nabla_{\bm x} \mc L\left(\bm x^\star, \bm \lambda^\star\left(\bm x^\star\right)\right)  = \bm 0 \label{eqn:first_order_opt}\\
& \mc{H}(\bm x^\star):=  \mc{P}_{\bm x^\star}^\top\nabla_{\bm x\bm x}^2 \mc  L\left(\bm x^\star, \bm \lambda^\star\left(\bm x^\star \right)\right)\mc{P}_{\bm x^\star} \succeq \mb 0
\label{eqn:second_order_opt}
\end{flalign}
where $\bm \lambda^\star(\bm x^\star) \in \arg \min_{\bm \lambda}\; \norm{\nabla_{\bm x} \mc L(\bm x^\star, \bm \lambda)}{}.$
\end{theorem}

  Feasible points satisfying conditions \eqref{eqn:first_order_opt} and \eqref{eqn:second_order_opt} are typically referred as {\em second-order critical points}. The gist of our paper is to design iterative numerical methods to locate them.  We next briefly describe examples of optimization problems that arise in signal processing, machine learning and statistics where such solutions are essentially sought.

\paragraph{Symmetric Orthogonal Tensor Decomposition (SOTD).} Tensor is a multidimensional array and the symmetric orthogonal tensor decomposition (SOTD) naturally generalizes the spectral decomposition of a symmetric matrix. Here, we focus on the $p$-way $n$-dimensional symmetric orthogonal (SOD \cite{mu2015successive, wang2017tensor, mu2017}) tensor
\begin{flalign}
\mcb T = \sum_{i=1}^n \underbrace{\bm v_i\otimes \bm v_i \otimes \cdots \otimes \bm v_i}_{p\; times} \in \reals^{\overbrace{n \times n \times \cdots \times n}^{p \; times}},
\end{flalign}
where $\mb V = [\bm v_1, \bm v_2, \ldots, \bm v_n] \in \reals^{n \times n}$ is an orthogonal matrix and $\otimes$ denotes the usual outer product, so the $(i_1, i_2, \ldots, i_p)$-th entry of $\bm v \otimes \bm v \otimes \cdots \otimes \bm v$ is the scalar $v_{i_1} v_{i_2}\cdots v_{i_p}$. The problem addressed by SOTD, is to find (up to sign and permutation) the components $\bm v_i$'s given $\mcb T$, has many applications, including higher-order statistical estimation \cite{mccullagh1987tensor}, independent component analysis \cite{comon1994independent, comon2010handbook}, and parameter estimation for latent variable models \cite{anandkumar2014tensor}).

  Ge et al. \cite{GHJY15} consider the SOTD specifically for the $p=4$ case and analyze the geometry of the following minimization problem
\begin{flalign}\label{eqn:tensor_opt}
\min_{\bm x \in \reals^n} \quad T(\bm x, \bm x, \bm x, \bm x) = \sum_{i=1}^n\sum_{j=1}^n\sum_{k=1}^n\sum_{l=1}^n \mc T_{ijkl} x_i x_j x_k x_l \qquad \mbox{s.t.} \quad \norm{\bm x}{2} = 1.
\end{flalign}
They prove that for problem \eqref{eqn:tensor_opt}, the points satisfying the second-order necessary condition coincide with the component $\bm v_i$'s. Specifically, they show that
\begin{flalign}
\set{\bm x \in \reals^n \; \vert \; G(\bm x) = 0,\; H(\bm x) \succeq \mb 0} = \set{\pm \bm v_1, \ldots, \pm \bm v_n}.
\end{flalign}
Therefore, solving problem \eqref{eqn:tensor_opt} yields one component $\bm {v}_i$, and after that, one can apply standard deflation procedures to obtain the others one by one.

  Ge et al. \cite{GHJY15} further propose a larger optimization problem
\begin{flalign}\label{eqn:tensor_opt_2}
\min_{[\bm x_1, \bm x_2, \ldots, \bm x_n] \in \reals^{n \times n}} \quad \sum_{i = 1}^n \sum_{ j \neq i} \mcb T(\bm x_i, \bm x_i, \bm x_j, \bm x_j)\qquad \mbox{s.t.} \quad \norm{\bm x_i}{2} = 1, \quad \; i = 1, \ldots, n,
\end{flalign}
to find all the components at one shot. Although problem \eqref{eqn:tensor_opt_2} is substantially more complicated than problem \eqref{eqn:tensor_opt}, similar geometrical phenomenon and the advantageous property of (\ref{eqn:tensor_opt}) are preserved. Specifically, Ge et al. \cite{GHJY15} proves that any solution of problem \eqref{eqn:tensor_opt_2} that satisfies the second-order necessary conditions in Theorem 1 is a signed and permuted version of $[\bm v_1, \bm v_2, \ldots, \bm v_n] \in \reals^{n \times n}$.
\paragraph{Semidefinite Programming (SDP).} SDP is one of the most exciting developments in mathematical optimization and has been successfully applied to model and solve problems in traditional convex constrained optimization, control theory, and combinatorial optimization. In general, an SDP problem can be defined as
\begin{flalign}\label{eqn:sdp_opt}
\min_{\mb X \in \mc S_+^{n \times n}}\quad \mbox{Tr}(\mb C \mb X) \qquad \mbox{s.t.} \quad \mbox{Tr}(\mb A_i \mb X) = b_i \quad \; i = 1, \ldots, m,
\end{flalign}
where $\mc S_+^{n \times n}$ is the cone of symmetric positive semidefinite matrices.

  Although interior-point methods can solve SDP problems in polynomial time, scalability is a problem and in practice interior-point methods run out of memory and time once $n$ is greater than $1e3$. To address this issue, Burer and Monteiro  \cite{burer2003nonlinear, burer2005local}, take advantage of the low-rank structure of SDP optimal solutions characterized by Pataki \cite{pataki1998rank} and Barvinok \cite{barvinok1995problems}. Explicitly they replaced $\mb X = \mb L \mb L^\top$ where $\mb L$ has size $n\times p$, which leads to the following non-convex but low-dimensional surrogate:
\begin{flalign}\label{eqn:sdp_opt_low_rank}
\min_{\mb L \in \reals^{n \times p}}\quad \mbox{Tr}(\mb L^\top \mb C \mb L) \qquad \mbox{s.t.} \quad \mbox{Tr}(\mb L^\top \mb A_i \mb L) = b_i \quad  \; i = 1, \ldots, m.
\end{flalign}
Burer and Monteiro  \cite{burer2003nonlinear, burer2005local} also applied local optimization methods to solve the above general nonlinear program \eqref{eqn:sdp_opt_low_rank} with surprisingly good performance, even though the theoretical justifications between problems \eqref{eqn:sdp_opt} and \eqref{eqn:sdp_opt_low_rank} were unclear.

  Recently, Boumal et al. \cite{boumal2016non} proved that whenever $p$ satisfies $p(p+1) > 2m$ (and some other technical but mild conditions hold), problem \eqref{eqn:sdp_opt_low_rank} is equivalent to problem \eqref{eqn:sdp_opt} in the sense that for any {\em second-order critical point} of \eqref{eqn:sdp_opt_low_rank} $\mb L^\star$, the square matrix $\mb L^\star (\mb L^\star)^\top$ is optimal to problem \eqref{eqn:sdp_opt}.

  One straightforward but powerful application of the above discussion is to solve the SDP problem derived as relaxation for the max-cut problem. In their celebrated work \cite{goemans1995improved}, Goemans and Williamson tackled the NP-hard max-cut problem by the following SDP
\begin{flalign}\label{eqn:sdp_opt_mc}
\min_{\mb X \in \mc S_+^{n \times n}}\quad \mbox{Tr}(\mb C \mb X) \qquad \mbox{s.t.} \quad \diag{\mb X} = \bm 1,
\end{flalign}
whose optimal solution (after rounding) will yield an approximate solution to max-cut within a ratio of .878. The corresponding Burer-Monteiro reformulation can be written as
\begin{flalign}\label{eqn:sdp_opt_low_rank_mc}
\min_{\mb L \in \reals^{n \times p}}\quad \mbox{Tr}(\mb L^\top \mb C \mb L) \qquad \mbox{s.t.} \quad \norm{\bm e_i^\top \mb L}{} = 1 \quad \; i = 1, \ldots, n.
\end{flalign}
Based on the above discussion, when we choose $p = \lceil \sqrt{2n} \rceil$, even though problem \eqref{eqn:sdp_opt_low_rank_mc} has substantially smaller dimension than problem \eqref{eqn:sdp_opt_mc}, any second-order critical point $\mb L^\star$ of \eqref{eqn:sdp_opt_low_rank_mc} is able to recover a global optimal solution of \eqref{eqn:sdp_opt_mc} with $\mb X^\star = \mb L^\star (\mb L^\star)^\top$.

\hspace{5mm}

  There is a large body of recent literature, especially in signal processing and learning theory, where finding second-order critical points is the fundamental target for the application at hand. Interested readers can find examples in complete dictionary recovery \cite{SQW152}, generalized phase retrieval \cite{SQW16}, matrix completion \cite{ge2016matrix}, phase synchronization  and community detection \cite{boumal2016nonconvex, bandeira2016low}.

\section{Equality-Constrained Problem}\label{sec:eq}
In this section, we extend the definitions \eqref{eqn:first_order_opt} and \eqref{eqn:second_order_opt} of $\mc{G}(\bm{x}^\star)$ and $\mc{H}(\bm x^\star)$ at a local minimizer $\bm x^\star$ to a general feasible point $\bm x\in \Omega$, and verify that $\mc{G}(\bm x)$ and $\mc{H}(\bm x)$ can be regarded as natural generalizations of the conventional gradient and Hessian. Based on this, we extend the classical negative curvature algorithm from the unconstrained problem to the equality constrained one \eqref{eqn:problem_to_solve}, and prove this method converges to a second-order critical point of \eqref{eqn:problem_to_solve}.

\subsection{Generalized gradient and Hessian}\label{sub:generalized}
Based on Theorem 1, we note that $\mc{G}(\bm x^\star)$ and $\mc{H}(\bm x^\star)$ provide an elegant characterization of local optimality for the equality constrained problem in analogy with the roles of the gradient and Hessian for the unconstrained one. Inspired by this, we extrapolate the definitions of \eqref{eqn:first_order_opt} and \eqref{eqn:second_order_opt} from constrained local minimizers $\bm x^\star$ to any feasible point $\bm x \in \Omega,$ with the intention of using these generalized quantities as one uses the gradient and Hessian in the algorithmic design for unconstrained optimizations.

\begin{definition}[Generalized Gradient and Hessian]
For any feasible $\bm x \in \Omega$, let
\begin{flalign}\label{eqn:lambda_def}
\bm \lambda^\star(\bm x) \in \arg \min_{\bm \lambda}\; \norm{\nabla_{\bm x} \mc L(\bm x, \bm \lambda)}{}.
\end{flalign}
Then we define the generalized gradient at $\bm x$ as
\begin{flalign}\label{eqn:G}
&\mc{G}(\bm x):=\nabla_{\bm x} \mc L\left(\bm x, \bm \lambda^\star\left(\bm x\right)\right)
\end{flalign}
and the generalized Hessian at $\bm x$ as
\begin{flalign}\label{eqn:H}
\mc{H}(\bm x):= \mc{P}_{\bm x}^\top\nabla_{\bm x\bm x}^2 \mc  L\left(\bm x, \bm \lambda^\star\left(\bm x\right)\right)\mc{P}_{\bm x}.
\end{flalign}
\end{definition}
  In the following, we verify that $\mc{G}(\cdot)$ and $\mc{H}(\cdot)$ behave like the gradient and Hessian not only for a local optimal solution but also for any feasible point $\bm x \in \Omega$ in an approximate sense, under the mild conditions described below:

\begin{assumption}
\label{ass: key_ass}
\begin{description}
\item[(a)]\hspace{0.4cm} $\nabla f$, $\nabla^2 f$, $\nabla c_i$ and $\nabla^2 c_i$ are Lipschitz continuous over $\Omega$ with Lipschitz constants $L_{f, 1}$, $L_{f, 2}$, $L_{c_i, 1}$ and $L_{c_i, 2}$.
\item[(b)]\hspace{0.4cm} $\sup_{{\bm x} \in \Omega}\norm{\nabla f({\bm x})}{} \leq \gamma_{f, 1}$, $\sup_{{\bm x}\in \Omega}\norm{\nabla^2 f({\bm x})}{} \leq \gamma_{f, 2}$, $\sup_{{\bm x} \in \Omega}\norm{\nabla c_i({\bm x})}{} \leq \gamma_{c_i, 1}$ and $\\\sup_{{\bm x} \in \Omega}\norm{\nabla^2 c_i({\bm x})}{} \leq \gamma_{c_i, 2}$.
\item[(c)][$\sigma_0$-LICQ] \hspace{0.4cm}For every ${\bm x} \in \Omega$, $\sigma_{\text{min}}(\nabla \mb{c}({\bm x})) \ge \sigma_0$ for some $\sigma_0 > 0$, where $\sigma_{\text{min}}(\nabla \mb{c}({\bm x}))$ is the smallest singular value of the matrix $\nabla \mb{c}(\bm{x})$
\end{description}
\end{assumption}
\begin{remark}
The last assumption is slightly stronger that LICQ as LICQ only requires $\sigma_{\text{min}}(\nabla \mb{c}({\bm x})) > 0.$
\end{remark}

  We next prove a key lemma regarding $\mc{G}(\cdot)$ and $\mc{H}(\cdot)$. Loosely speaking, we show that for every $\mb{x} \in \Omega$, given a small perturbation $\bm \delta \in \mc{T}_\Omega(\mb{x})\subset  \reals^n$, $f({\bm x})+\mc{G}({\bm x})^\top{\bm \delta}$ and $f({\bm x})+\mc{G}({\bm x})^\top{\bm \delta} +\frac{1}{2}{\bm \delta}^\top \mc{H}({\bm x}){\bm \delta}$, respectively approximates $f(\bm x + \bm \delta)$ up to the first order and the second order. This result is crucial for our generalization of negative curvature methods. But one issue we need to fix in advance is that $\bm x + \bm \delta$ may possibly lie outside of $\Omega$ and thus the objective $f$ might not even be defined at $\bm x + \bm \delta$.  In order to resolve this infeasibility issue, we introduce the following projection operator:
\begin{definition}\label{def:proj_Omega}
For any $\bm y \in \reals^n$, we denote by $\Pi_{\Omega}(\bm y)\in \Omega$ a point that is closest to $\bm y$, i.e. $\Pi_{\Omega}(\bm y)\in \arg\min_{\bm x \in \Omega} \; \norm{\bm x - \bm y}{}$, where $\norm{\cdot}{}$ denotes the Euclidean norm.
\end{definition}


  We are now ready to state the key lemma.

\begin{lemma}\label{key_lemma}
Under Assumptions 2(b), 2(c) and 2(d),
for any ${\bm x}_0 \in \Omega$ and ${\bm \delta} \in \mc{T}_{\Omega}(\bm x_0)$, we have
\begin{flalign}
&\left|f(\Pi_\Omega({\bm x}_0 + {\bm \delta})) - f({\bm x}_0) -\mc{G}({\bm x}_0)^\top{\bm \delta} \right| \le C_0 \norm{{\bm \delta}}{}^2 ,\quad\text{and} \label{1st_order_taylor} \\
&\left|f(\Pi_\Omega({\bm x}_0 + {\bm \delta})) - f({\bm x}_0) -\mc{G}({\bm x}_0)^\top{\bm \delta} - \frac{1}{2}{\bm \delta}^\top \mc{H}({\bm x}_0){\bm \delta}\right| \le C_5 \norm{{\bm \delta}}{}^3 \label{key_taylor}
\end{flalign}
where $C_0 = \gamma_{f, 1}\{1 + \frac{\Gamma_1}{\sigma_0^2}\}\frac{4}{R^2} + 4(L_{f, 1} + \frac{\gamma_{f, 1}\sqrt{\Gamma_1\Lambda_1}}{\sigma_0^2} )$, $C_5 = 8C_1 + C_2 + C_3 + C_4, C_1 = \frac{L_{f, 2}}{2} + \frac{\gamma_{f, 1}}{2\sigma_0^2}\sqrt{\Gamma_1\Lambda_2}$, $C_2 = \{\gamma_{f, 2} + \frac{\gamma_{f, 1}}{2\sigma_0^2}\sqrt{\Gamma_1\Gamma_2}\}\frac{4}{R}$, $C_3 = \{\gamma_{f, 2} + \frac{\gamma_{f, 1}}{2\sigma_0^2}\sqrt{\Gamma_1\Gamma_2}\}\frac{2}{R}$, $C_4 = \{\gamma_{f, 1} + \frac{\gamma_{f, 1}}{\sigma_0^2}\Gamma_1\}(\frac{2\sqrt{\Gamma_1\Lambda_1}}{\sigma_0^2} + \frac{2\sqrt{\Gamma_1^3\Lambda_1}}{\sigma_0^4}) \frac{8}{R}$, $R = \sigma_0 / \sqrt{\Lambda_1}$,  $\Gamma_1 = \sum_{i=1}^m \gamma_{c_i, 1}^2$, $\Gamma_2 = \sum_{i=1}^m \gamma_{c_i, 2}^2$, $\Lambda_1 = \sum_{i=1}^m L_{c_i, 1}^2$, and $\Lambda_2 = \sum_{i=1}^m L_{c_i, 2}^2$.
\end{lemma}
\begin{proof}
See the appendix.
\end{proof}

\begin{remark}
If $\Omega$ is a differentiable manifold, our $\mc{G}(\cdot)$ and $\mc{H}(\cdot)$ defined in \eqref{eqn:G} and \eqref{eqn:H} can be proven to be the Riemannian gradient and Riemannian Hessian of $f$ over $\Omega$ (see subsection \ref{sub:relations} for details). The expressions \eqref{eqn:G} and \eqref{eqn:H} concretize these abstract geometrical concepts by providing an explicit algebraic way to compute them.
\end{remark}

\subsection{Negative curvature method}
In this subsection, we will present a general framework for using negative curvature of the generalized Hessian $\mc{H}(\bm x)$ to solve (\ref{eqn:problem_to_solve}) and show that the sequence generated by this framework converges to a second-order critical point.

  It is clear that using the simple projected gradient method
\begin{align}
\bm x_{k+1} \gets \Pi_\Omega\left({\bm x}_k - t \mc{G}(\bm x_k)\right)  \label{project_gradient_step}
\end{align}
might be trapped by the saddle points of \eqref{eqn:problem_to_solve}.
To overcome this difficulty, it is natural to consider using second-order information, that is, information about the Hessian of $f$ and the $c_i$'s. In particular, we can use a negative curvature direction of $\mc{H}(\bm x)$ defined in subsection \ref{sub:generalized}. We say $\mb{d}$ is a negative curvature direction of $\mc{H}(\bm x)$ if it has the property that $\mb{d}^\top\mc{H}(\bm x)\mb{d} < 0$. Based on the second-order approximation result revealed in Lemma \ref{key_lemma}, it is intuitive that moving in a negative curvature direction will enable an algorithm to escape from a saddle point. Therefore, it is tempting to move along a direction that combines the negative gradient direction with a descent direction of negative curvature:
\begin{align}
\bm x_{k+1} \gets \Pi_\Omega\left({\bm x}_k - t_1 \mc{G}(\bm x_k) - t_2 \bm d \right) \label{project_curvilinear_step}
\end{align}
especially for the regions close to saddle points.


  Moreover, negative-curvature directions can be obtained at a relatively small cost. For example,  when a Hessian matrix $H(\bm x)$ is indefinite, the eigenvector corresponding to its algebraically smallest eigenvalue is a negative curvature direction. It can be obtained by executing the power iteration method on $H(\bm x)$ to obtain a eigenvalue, eigenvector pair $(\lambda_{\text{dom}}, \bm{d}_{\text{dom}})$ corresponding to the eigenvalue of largest magnitude, the so-called {\it dominant} eigenvalue $\lambda_{\text{dom}}$. If $\lambda_{\text{dom}} < 0$, $\bm d_{\text{dom}}$ is the direction of most negative curvature. Otherwise, we can perform the method again on $\mc{H} - \lambda_{\text{dom}}\mc{I}$. More efficient and robust methods using variants of Lanzos algorithm to compute the algebraically smallest eigenpair can be found in \cite{larsen2004propack, stathopoulos2010primme}.

  In Algorithm \ref{curvlinear_search}, we provide a general framework for using negative curvature directions to solve problem \eqref{eqn:problem_to_solve}. This algorithm integrates the first-order and second-order methods. Specifically, when the iterate $\bm x_k$ is far from any saddle point, we only use the first-order information to make progress:
\[
\hat{{\bm x}} \leftarrow \Pi_\Omega({\bm x}_k - t\mcb{G}_k).
\]
When $\bm x_k$ is near saddle points, we combine the negative gradient and the negative curvature direction:
\begin{flalign}
\hat{{\bm x}} \leftarrow \Pi_\Omega({\bm x}_k - t\mcb{G}_k + t^\alpha \mb{d}_k).
\label{eqn:cur_search}
\end{flalign}
If $\alpha$ is chosen to be $2$ and $\Omega = \reals^n$ , \eqref{eqn:cur_search} is reduced to the one leveraged in \cite{GOLD80}; if $\alpha$ is set to $1/2$ and $\Omega = \reals^n$, \eqref{eqn:cur_search} is equivalent to the one used in \cite{MS79}.

In the rest of this section, we will confirm our intuition that Algorithm \ref{curvlinear_search} is capable of escaping saddle points by proving the following theorem:

\begin{center}
\begin{algorithm}[h]
\caption{Negative Curvature Method for Problem \eqref{eqn:problem_to_solve}}
\label{curvlinear_search}
\begin{algorithmic}[1]
\renewcommand\algorithmicrequire{\textbf{input}}
\Require parameters $0 < \sigma < 1$, $0< \rho < 1$, $\alpha > 0$, $\epsilon > 0$ and $t_0 > 0$.
\State initialize ${\bm x}_0 \in \Omega$;
\For{$k = 0, 1, \ldots, $}
\State $\mcb{G}_k \leftarrow \mc{G}(\mb{x_k})$
\If{$\|\mcb{G}_k\| \geq \epsilon$}
\State $t \leftarrow t_0$
\State $\hat{{\bm x}} \leftarrow \Pi_\Omega({\bm x}_k - t\mcb{G}_k)$
\While{$f(\hat{{\bm x}}) - f({\bm x}_k) > -\sigma t \|\mcb{G}_k\|^2$}
\State $t \leftarrow \rho t$
\State $\hat{{\bm x}} \leftarrow \Pi_\Omega({\bm x}_k - t\mcb{G}_k)$
\EndWhile
\Else
\State $\mcb{H}_k \leftarrow \mc{H}({\bm x}_k)$
\State $(\lambda_k^\text{min}, \mb{v}_k) \leftarrow$ the algebraically smallest eigenpair of $\mcb{H}_k$
\State $\lambda_k \leftarrow \min\{\lambda_k^\text{min}, 0\}$
\State $\mb{d}_k \leftarrow \vert \lambda_k \vert \text{sign}(-\mb{v}_k^\top\mcb{G}_k)\mb{v}_k$
\State $t \leftarrow t_0$
\State
 $\hat{{\bm x}} \leftarrow \Pi_\Omega({\bm x}_k - t\mcb{G}_k + t^\alpha \mb{d}_k)$
\While{$f(\hat{{\bm x}}) - f({\bm x}_k) > \sigma \big( -t\|\mcb{G}_k\|^2 - \frac{1}{2}t^{2\alpha}\vert\lambda_k\vert^3\big)$}
\State $t \leftarrow \rho t$
\State $\hat{{\bm x}} \leftarrow \Pi_\Omega({\bm x}_k - t\mcb{G}_k + t^\alpha \mb{d}_k)$
\EndWhile
\EndIf
\State $t_k \leftarrow t$
\State ${\bm x}_{k+1} \leftarrow \hat{{\bm x}}$
\EndFor
\end{algorithmic}
\end{algorithm}
\end{center}

\begin{theorem}\label{thm:alg}
For the sequences $\{\mcb{G}_k\}$ and $\{\lambda_k\}$ generated by Algorithm \ref{curvlinear_search}, one has $\mcb{G}_k \rightarrow \mb{0}$ and $\lambda_k \rightarrow 0$ as $k \rightarrow \infty$.
\end{theorem}

\begin{remark}
Based on Theorem \ref{thm:alg}, any cluster point of $\set{\bm x_k}$ generated by Algorithm \ref{curvlinear_search} is a second-order critical point of problem \eqref{eqn:problem_to_solve}.
\end{remark}

  To prove Theorem \ref{thm:alg}, we need several lemmas.

\begin{lemma} \label{expan_lemma}
Consider the $k$-th iteration.
\begin{description}
\item[(a)]For the parametrized curve ${\bm x}(t) = \Pi_\Omega({\bm x}_k - t\mcb{G}_k)$, we have
\begin{align}
f({\bm x}(t)) - f({\bm x}_k) \leq -\|\mcb{G}_k\|^2t + C_0\|\mcb{G}_k\|^2t^2;
\label{eqn:first_ineq}
\end{align}
\item[(b)]For the parametrized curve ${\bm x}(t) = \Pi_\Omega({\bm x}_k - t\mcb{G}_k + t^\alpha\mb{d}_k)$, we have
\begin{align}
f({\bm x}(t)) - f({\bm x}_k) \leq -\|\mcb{G}_k\|^2t + \frac{1}{2}\|\mcb{H}_k\|\|\mcb{G}_k\|^2t^2 - \frac{1}{2}\vert \lambda_k\vert^3 t^{2\alpha} + 8C_5\|\mcb{G}_k\|^3t^3+ 8C_5\vert\lambda_k\vert^3t^{3\alpha}.
\label{eqn:second_ineq}
\end{align}
\end{description}

\end{lemma}
\begin{proof}
For part (a), we can directly apply (\ref{1st_order_taylor}) in Lemma \ref{key_lemma}:
\begin{flalign}
f({\bm x}(t)) - f({\bm x}_k) \leq \innerprod{\mcb{G}_k}{-t\mcb{G}_k}+ C_0\|\mcb{G}_k\|^2t^2 = -\|\mcb{G}_k\|^2t + C_0\|\mcb{G}_k\|^2t^2.
\nonumber
\end{flalign}

  Now let us focus on part (b).

  For part (b), based on \eqref{key_taylor} of Lemma \ref{key_lemma}, we first have
\begin{flalign}
&f({\bm x}(t)) - f({\bm x}_k) \nonumber\\
&\leq \langle \mcb{G}_k, -t\mcb{G}_k + t^\alpha \mb{d}_k\rangle + \frac{1}{2}(-t\mcb{G}_k + t^\alpha \mb{d}_k)^\top \mcb{H}_k (-t\mcb{G}_k + t^\alpha \mb{d}_k) + C_5\|-t\mcb{G}_k + t^\alpha \mb{d}_k\|^3 \nonumber\\
&\leq -\|\mcb{G}_k\|^2 t + \langle \mcb{G}_k, \mb{d}_k \rangle t^\alpha + \frac{1}{2}\mcb{G}_k^\top \mcb{H}_k\mcb{G}_kt^2 - \mcb{G}_k^\top \mcb{H}_k \mb{d}_kt^{\alpha + 1} + \frac{1}{2}\mb{d}_k^\top\mcb{H}_k\mb{d}_kt^{2\alpha} + C_5\|-t\mcb{G}_k + t^\alpha \mb{d}_k\|^3. \label{eqn:total}
\end{flalign}
Next, we will bound each term in \eqref{eqn:total}:
\begin{flalign}
&\langle \mcb{G}_k, \mb{d}_k \rangle t^\alpha \leq |\lambda_k| \mbox{sign}(-\bm v_k^\top \mcb{G}_k)\bm v_k^\top \mcb{G}_k t^\alpha \le 0 \label{eqn:sub_1}\\
&\frac{1}{2} \mcb{G}_k^\top \mcb{H}_k \mcb{G}_k t^2 \leq \frac{1}{2} \norm{\mcb{H}_k}{} \norm{\mcb{G}_k}{}^2 t^2 \label{eqn:sub_2}\\
&-\mcb{G}_k^\top \mcb{H}_k \bm d_k t^{\alpha+1} = - \mcb{G}_k^\top |\lambda_k| \mbox{sign}(-\bm v_k^\top \mcb{G}_k) \mcb{H}_k \bm v_k t^{\alpha+1} = \lambda_k |\lambda_k| \mbox{sign}(-\bm v_k^\top \mcb{G}_k) (- \mcb{G}_k^\top\bm v_k) t^{\alpha+1} \le 0 \label{eqn:sub_3}\\
&\frac{1}{2} \bm d_k^\top \mcb{H}_k \bm d_k t^{2\alpha} = \frac{1}{2} |\lambda_k|^2 \bm v_k^\top \mcb H_k \bm v_k  t^{2\alpha} = -\frac{1}{2} |\lambda_k|^3 t^{2 \alpha} \label{eqn:sub_4}   \\
&\|-t\mcb{G}_k + t^\alpha \mb{d}_k\|^3 \le \left( t\norm{\mcb{G}_k}{} + t^\alpha \norm{\bm d_k}{}   \right)^3 \le 8 \max(t \norm{\mcb G_k}{}, t^\alpha \norm{\bm d_k}{})^3 \le 8 t^3 \norm{\mcb G_k}{}^3 + 8 t^{3 \alpha} |\lambda_k|^3. \label{eqn:sub_5}
\end{flalign}
With \eqref{eqn:sub_1}-\eqref{eqn:sub_5} plugged in \eqref{eqn:total}, we reach \eqref{eqn:second_ineq}.

\begin{lemma}
There exists a constant $\gamma_h \ge 0$ such that $\sup_{\bm x \in \Omega} \norm{\mc H(\bm x)}{} \le \gamma_h.$
\end{lemma}

\begin{proof}
For any $\bm x \in \Omega$, one has
\begin{flalign}
\norm{\mc H(\bm x)}{} &\le \norm{\mc P_{\bm x}^\top \nabla_{\bm x \bm x}^2 \mc L(\bm x, \bm \lambda^\star (\bm x))   \mc P_{\bm x}}{} \nonumber \\
& \le \norm{ \nabla_{\bm x \bm x}^2 \mc L(\bm x, \bm \lambda^\star (\bm x))}{} \nonumber \\
& \le \norm{\nabla_{}^2 f(\bm x) - \sum_{i \in [m]} \lambda_i^\star (\bm x) \nabla^2 c_i(\bm x)}{} \nonumber \\
& \le \norm{\nabla^2 f(\bm x)}{} + \norm{\lambda^\star (\bm x)}{\infty} \sum_{i \in [m]} \norm{\nabla^2 c_i(\bm x)}{}. \label{H_bound}
\end{flalign}

Based on Assumption \ref{ass: key_ass}, we have
\begin{flalign}
\norm{\nabla^2 f(\bm x)}{} \le \gamma_{f,2} \;\; \mbox{and} \;\; \sum_{i \in [m]} \norm{\nabla^2 c_i(\bm x)}{} \le \sum_{i \in [m]} \gamma_{c_i, 2}.
\label{eqn:first_bd}
\end{flalign}

From the definition \eqref{eqn:lambda_def}. we can derive that
\begin{flalign}
\bm \lambda^\star({\bm x}) = ({\nabla \bm c}({\bm x})^\top{\nabla\bm c}({\bm x}))^{-1}{\nabla \bm c}({\bm x})^\top\nabla f({\bm x}). \label{lambda_express}
\end{flalign}
Together with Assumption \ref{ass: key_ass}, it can be obtained that
\begin{flalign}
\norm{\bm \lambda^\star({\bm x})}{\infty} & \le \norm{\bm \lambda^\star({\bm x})}{2} \nonumber\\
& = \norm{({\nabla \bm c}({\bm x})^\top{\nabla\bm c}({\bm x}))^{-1}{\nabla \bm c}({\bm x})^\top\nabla f({\bm x})}{2} \nonumber \\
& \le \norm{({\nabla \bm c}({\bm x})^\top{\nabla\bm c}({\bm x}))^{-1}}{} \norm{\nabla \bm c(\bm x)}{} \norm{\nabla f(\bm x)}{} \nonumber \\
& \le \norm{({\nabla \bm c}({\bm x})^\top{\nabla\bm c}({\bm x}))^{-1}}{} \norm{\nabla \bm c(\bm x)}{F} \norm{\nabla f(\bm x)}{} \nonumber \\
& \le \frac{1}{\sigma_0^2} \sqrt{\sum_{i \in [m]} \gamma^2_{c_i,1}} \cdot \gamma_{f,1}.
\label{eqn:lam_inf}
\end{flalign}
The lemma can be established by substituting \eqref{eqn:first_bd} and \eqref{eqn:lam_inf} into (\ref{H_bound}). Hence $\gamma_h = \sum_{i\in [m]}\gamma_{c_i, 2} + \frac{1}{\sigma_0^2} \sqrt{\sum_{i \in [m]} \gamma^2_{c_i,1}} \cdot \gamma_{f,1}\sum_{i \in [m]} \gamma_{c_i, 2}$
\end{proof}

%
%
%

\end{proof}
\begin{lemma}
$\{t_k\}$ is uniformly bounded from below, i.e., there exists $\beta > 0$ such that $t_k \geq \beta$ for every $k \in \mr{N}$. \label{lem:t_k}
\end{lemma}
\begin{proof}
Consider the first case $\|\mcb{G}_k\| \ge \epsilon$. For ${\bm x}(t) = \Pi_\Omega({\bm x}_k - t\mcb{G}_k)$, as shown in Lemma \ref{expan_lemma}, one has
\begin{flalign}
f({\bm x}(t)) - f({\bm x}_k) \leq -\|\mcb{G}_k\|^2t + C_0 \|\mcb{G}_k\|^2t^2.
\end{flalign}
Therefore, whenever $t \le (1-\sigma)/C_0$, we have
\begin{flalign}
f({\bm x}(t)) - f({\bm x}_k) \leq -\|\mcb{G}_k\|^2t + C_0 \|\mcb{G}_k\|^2t^2 \le -\sigma t \norm{\mcb G_k}{}^2.
\end{flalign}
Thus,  $t_k \geq t_0 \rho^{\lceil \log_\rho\{(1 - \sigma) / (C_0 t_0)\}\rceil}$ for this case.

  Now let us consider the other case $\|\mcb{G}_k\| < \epsilon$. For ${\bm x}(t) = \Pi_\Omega({\bm x}_k - t\mcb{G}_k + t^\alpha\mb{d}_k)$, it follows from part (b) of Lemma \ref{expan_lemma} that
\begin{align}
f({\bm x}(t)) - f({\bm x}_k) \leq -\|\mcb{G}_k\|^2t + \frac{1}{2}\|\mcb{H}_k\|\|\mcb{G}_k\|^2t^2 - \frac{1}{2}\vert \lambda_k\vert^3 t^{2\alpha} + 8C_5\|\mcb{G}_k\|^3t^3+ 8C_5\vert\lambda_k\vert^3t^{3\alpha}.
\end{align}
When $t \le \underline{t} := \min \set{\left(\frac{1 - \sigma}{16C_5}\right)^{1/\alpha}, \;\frac{2 - 2\sigma}{\gamma_h + 16C_5\epsilon}, \;1}$, it can be verified that
\begin{align}
-\frac{1}{2}\vert\lambda_k\vert^3t^{2\alpha} + 8C_5\vert\lambda_k\vert^3t^{3\alpha}&\leq -\frac{1}{2}\sigma t^{2\alpha}\vert \lambda_k \vert^3 \quad
\text{and} \label{first_ieq} \\
-\|\mcb{G}_k\|^2t + \frac{1}{2}\|\mcb{H}_k\|\|\mcb{G}_k\|^2t^2 + 8C_5\|\mcb{G}_k\|^3t^3 &\leq -\sigma t\|\mcb{G}_k\|^2\label{second_ieq}
\end{align}
Combining \eqref{first_ieq} and \eqref{second_ieq}, we have
\begin{align*}
f({\bm x}(t)) - f({\bm x}_k) \leq \sigma \big(-t\|\mcb{G}_k\|^2 - \frac{1}{2}t^{2\alpha}\vert\lambda_k\vert^3\big).
\end{align*}
Therefore,
$
t_k \geq t_0 \rho^{\lceil \log_\rho \underline t/ t_0 \rceil}
$
for this case.

  Taking both cases into consideration, we have proved this lemma.

%
%
%
\end{proof}


  Now we are ready to prove Theorem \ref{thm:alg}.

\begin{proof}[Proof of Theorem \ref{thm:alg}]
Based on Lemma \ref{lem:t_k}, we have
\begin{align*}
f({\bm x}_{k+1}) - f({\bm x}_k) \leq -\sigma \{\beta \epsilon^2, \beta\|\mcb{G}_k\|^2 + \frac{1}{2} \beta^{2\alpha}\vert\lambda_k\vert^3\}.
\end{align*}
Since $\lim_{k\rightarrow\infty} f({\bm x}_k) > -\infty$, we must have
\begin{align*}
\lim_{k\rightarrow \infty}\mcb{G}_k =\mb{0}, \quad\text{and}\quad \lim_{k\rightarrow\infty}\lambda_k = 0.
\end{align*}
\end{proof}

\begin{remark}
In Algorithm \ref{curvlinear_search}, we can relax the requirement of finding the algebraically smallest eigenpair of $\mcb{H}_k$. Instead, it is sufficient to find $(\bar \lambda, \bar{\bm v})$ satisfying
\begin{flalign}
\norm{\bar{\bm v}}{}= 1, \; \bar{\bm v}^\top \mcb{H}_k \bar{\bm v} &\le \max\{-\delta, \;\lambda_k^{\min}\}, \; \bar{\bm v}^\top \mcb{G}_k \le 0 \;\; \mbox{and} \;\; \bar{\bm v}^\top \mcb{H}_k \bar{\bm v} \le 0
\end{flalign}
where $\delta >0$ is a prescribed constant.
\end{remark}

\subsection{Examples of $\Omega$ and $\Pi_{\Omega}(\cdot)$}

For Algorithm 1 to be practical, the projection of a point onto the feasible set $\Omega$ must be affordable.
In this subsection, we enumerate a number feasible sets $\Omega$, defined by equality constraints, that are frequently encountered in optimization problems and have computationally tractable projections $\Pi_{\Omega}(\cdot)$:

\begin{table}[ht]
\caption{Examples of $\Omega$ and $\Pi_{\Omega}(\cdot)$} 
\centering 
\begin{tabular}{l |c | c } 
\hline
Constraint Sets & $\Omega$ & $\Pi_{\Omega}(\cdot)$  \\ [1ex] 
\hline 
spherical &$\set{\mb X \in \reals^{n \times m} \;\vert\; \norm{\mb X}{F} = 1 }$ & $\mb X / \norm{\mb X}{F}$  \\[1ex]
multiple spherical & $\bigcup_i \{ \mb x_i \in \reals^{n_i}  \;\vert\; \norm{\mb x_i}{2} = 1 \} $ & $ \mb x_i / \norm{\mb x_i}{2} , ~ \forall ~ i $
\\ [1ex]
orthogonality & $\set{ \mb X \in \reals^{n \times m} \;\vert\; \mb X^T \mb X =I }$ &
$\mb U \mb V^\t, ~\mbox{with}~ \mb X = \mb U \mb \Sigma \mb V^\t \mbox{ as SVD}$
\\ [1ex]

\hline 
\end{tabular}
\label{table:ex_omega} 
\end{table}



  There is a wide range of applications in which one is interested in solving optimization problems with constraints as listed in Table \ref{table:ex_omega}. The paper \cite{wen2013feasible} provides an extensive list of such problems and references to particular applications. These include eigenvalue and subspace tracking problems arising in signal processing; low-rank matrix optimization problems such as those that arise in SDP relaxation of combinatorial problems (e.g., the max-cut problem described in section 2); $p$-harmonic flows and other problems involving normal preserving constraints such as those that arise in $1$-bit compressive sensing, color image denoising, micromagnetics, liquid crystal theory, and directional diffusion; homogeneous polynomial optimization with spherical constraints arising in tensor eigenvalue problems signal processing, MRI, data training, approximation theory, portfolio selection and computation of the stability number of a graph; sparse principal component analysis, electronic structures computation, etc.
\subsection{Discussion: extension to general constrained problems}
We briefly discuss possibilities of generalizing the proposed algorithm to optimization problems with inequality constraints, i.e., problems of the form:
\begin{align*}
\begin{array}{ll}
\text{minimize} & f({\bm x}) \\
\mbox{subject to}& c_i({\bm x}) = 0, \; i = 1, 2, \ldots, m\\
 & c_i({\bm x}) \leq 0, \; i = m+1, 2, \ldots, k\\
& {\bm x} \in \mr{R}^n.
\end{array}
\end{align*}
  By adding a squared slack variable to each inequality constraint, since $c_i({\bm x}) \leq 0 \Leftrightarrow c_i({\bm x}) + z_i^2 = 0$, one can transform a problem
with inequality constraints to one that has only equality constraints. Moreover,
it can be verified that the conditions imposed in Assumption 2 on the constraints $c_i({\bm x})$, carry over to the transformed constraints. Therefore, algorithm \ref{curvlinear_search} can be applied again. But one caveat of this reformulation is that additional second-order critical points might be introduced. We leave it as future work to investigate better approaches to handling general inequality constraints.

\section{Appendix}\label{sec:append}
\subsection{Proof of Lemma \ref{key_lemma}}

We need the proposition below in the proof.
\begin{proposition}\label{the_prop}(\cite[Lemma 33]{GHJY15})
Assume that Assumption 2 holds and define $R = \sqrt{1 /  (\sum_{i=1}^m L_{c_i,1}^2/\sigma_0^2)} $. For any ${\bm x}_0 \in \Omega$, any $\mb{v} \in \mr{R}^m$, let ${\bm x}_1 = {\bm x}_0 + \mb{v}$ and ${\bm x}_2 = {\bm x}_0 + \mc{P}_{{\bm x}_0}(\mb{v})$. Then we have
\begin{align*}
\norm{\Pi_\Omega({\bm x}_1) - {\bm x}_2} \leq \frac{4\|\mb{v}\|^2}{R}.
\end{align*}
\end{proposition}

  Now we are ready to prove Lemma \ref{key_lemma}.
\begin{proof}[Proof of Lemma \ref{key_lemma}]
To prove (\ref{1st_order_taylor}), for any $\bm x \in \Omega$, let $\bm y = \Pi_\Omega(\bm x + \bm\delta)$ and consider the Lagrangian $\mc{L}(\bm y, \bm\lambda^*(\bm x)) = f(\bm y) - \sum_{i=1}^m \lambda^*_i(\bm{x})c_i(\bm y)$. Since $\bm x, \bm y \in \Omega$, $\bm c(\bm x) = \bm c (\bm y) = \bm 0$ and hence $\mc{L}(\bm y, \bm\lambda^*(\bm x)) = f(\bm y)$ and $\mc{L}(\bm x, \bm\lambda^*(\bm x)) = f(\bm x)$. Therefore by Taylor expansion, we have for some $s \in (0, 1)$
\begin{align}
f(\bm y) &= \mc{L}(\bm y, \bm\lambda^*(\bm x)) = \mc{L}(\bm x, \bm\lambda^*(\bm x)) + \nabla_{\bm x}\mc{L}(\bm x + s(\bm y - \bm x), \bm\lambda^*(\bm x))^\top (\bm y - \bm x) \nonumber\\
&= f(\bm x) + \nabla_{\bm x} \mc{L}(\bm x, \bm \lambda^*(\bm x))^\top (\bm y - \bm x) + \{\nabla_{\bm x}\mc{L}(\bm x + s(\bm y - \bm x), \bm\lambda^*(\bm x)) - \nabla_{\bm x} \mc{L}(\bm x, \bm \lambda^*(\bm x))\}^\top (\bm y - \bm x) \nonumber \\
&= f(\bm x) + \mc{G}(\bm x)^\top \bm \delta + \mc{G}(\bm x)^\top\{\bm y - (\bm x + \bm \delta)\} \nonumber \tag{by definition of $\mc{G}(\bm x)$} \\
&\quad +\{\nabla f(\bm x + s(\bm y - \bm x)) - \nabla f(\bm x)\}^\top (\bm y - \bm x) - \sum_{i=1}^m\lambda_i^*(\bm x)\{\nabla c_i(\bm x + s(\bm y - \bm x)) - \nabla c_i(\bm x)\}^\top (\bm y - \bm x). \nonumber
\end{align}
Hence using Assumption 2(b), Cauchy-Schwartz and the fact that $s < 1$, we have
\begin{align}
\vert f(\bm y) -  f(\bm x) - \mc{G}(\bm x)^\top \bm \delta\vert &\leq \|\mc{G}(\bm x)\|\|\bm y - (\bm x + \bm \delta)\| + \{L_{f, 1} + \sum_{i=1}^m\lambda^*_i(\bm x)L_{c_i, 1} \} \| \bm y - \bm x \|^2  \nonumber\\
&\leq \|\mc{G}(\bm x)\|\|\bm y - (\bm x + \bm \delta)\| + \{L_{f, 1} + \|\bm\lambda^*(\bm x)\|\sqrt{\Lambda_1}\}\| \bm y - \bm x \|^2 . \label{first_order_taylor}
\end{align}
Both $\| \bm\lambda^*(\bm x) \|$ and $\|\mc{G}(\bm x) \|$ can be bounded by constants. From (\ref{lambda_express}), use of the assumptions that $\inf_{\bm x \in \Omega}\sigma_{\text{min}}(\nabla\bm{c}({\bm x})) \ge \sigma_0 $, $\sup_{\bm x \in \Omega} \|\nabla f(\bm x)\| \leq \gamma_{f, 1}$, $\sup_{\bm x \in \Omega} \|\nabla c_i(\bm x)\| \leq \gamma_{c_i, 1}$ and the definition of Frobenius norm,
\begin{align}
\|\boldsymbol \lambda^\star(\boldsymbol x)\| \leq \frac{\gamma_{f, 1}\sqrt{\Gamma_1}}{\sigma_0^2} \label{lambda_bound}
\end{align}
While
\begin{align}
\| \mc{G}(\bm x) \| &= \|\nabla f(\bm x) - \sum_{i=1}^m \lambda_i^*(\bm x) \nabla c_i(\bm x) \| \leq \| \nabla f(\bm x) \| + \| \bm \lambda^*(\bm x) \| \| \nabla \bm c(\bm x) \|_F \nonumber \\
&\leq \gamma_{f, 1} + \frac{\gamma_{f, 1}\Gamma_1}{\sigma_0} = \gamma_{f, 1}\{1 + \frac{\Gamma_1}{\sigma_0^2}\}. \label{G_bound}
\end{align}
To bound $\| \bm y - \bm x\|$, by definition of $\bm y$ and $\Pi_\Omega(\cdot)$ we have
\begin{align}
\| \bm y - \bm x \| = \| \Pi_\Omega(\bm x + \delta) - \bm x \| \leq \| \Pi_\Omega(\bm x + \delta) - (\bm x + \bm \delta)\| + \| \bm \delta \| \leq 2\|\bm \delta \| \label{diff_bound}
\end{align}
As above $\|\bm y - (\bm x + \bm \delta) \| \leq \| \bm\delta \|$. However, we also need to bound $\|\bm y - (\bm x + \bm \delta) \|$ in terms of $\| \bm \delta \|^2$ to facilitate our analysis. Specifically, since $\bm \delta \in \mc{T}_{\Omega}(\bm x)$ and $\bm x \in \Omega$, by Proposition 1, we have
\begin{align}
\|\bm y - (\bm x + \bm \delta) \| = \|\Pi_\Omega(\bm x + \bm \delta) -(\bm x + \bm \delta) \| \leq \frac{4}{R^2}\| \bm\delta\|^2 \label{tricki_bound}
\end{align}
where $R = \sigma_0 / \sqrt{\Lambda_1}$. Now plugging (\ref{tricki_bound}), (\ref{diff_bound}), (\ref{G_bound}), and (\ref{lambda_bound}) into (\ref{first_order_taylor}) we obtain the desired result (\ref{1st_order_taylor}), that is,
\begin{align*}
| f(\bm y) - f(\bm x) - \mc{G}(\bm x)^\top   \bm \delta  | \leq C_0 \| \bm \delta \|^2.
\end{align*}
where $C_0 = \gamma_{f, 1}\{1 + \frac{\Gamma_1}{\sigma_0^2}\}\frac{4}{R^2} + 4(L_{f, 1} + \frac{\gamma_{f, 1}\sqrt{\Gamma_1\Lambda_1}}{\sigma_0^2} )$.

  To prove (\ref{key_taylor}), for every ${\bm x, \bm y} \in {\Omega}$, by the definition of the Lagrangian, we have $f({\bm x}) = \mc{L}({\bm x}, \boldsymbol\lambda^\star({\bm x}))$ and $f(\bm{y}) = \mathcal{L}(\bm{y}, \boldsymbol\lambda^\star({\bm x}))$. Let $\bm \eta = {\bm y} - {\bm x}$. Then by Taylor's theorem, there exists a $t \in (0, 1)$ such that
\begin{align}
f({\bm y}) = \mc{L}({\bm y}, \boldsymbol\lambda^\star({\bm x})) &= \mathcal{L}({\bm x}, \boldsymbol\lambda^\star({\bm x})) + \nabla_{\bm x}\mathcal{L}({\bm x}, \boldsymbol\lambda^\star({\bm x}))^\top{\bm \eta} + \frac{1}{2}{\bm \eta}^\top \nabla_{\mathbf{xx}}^2\mc{L}({\bm x} + t{\bm \eta}, \boldsymbol\lambda^\star({\bm x})){\bm \eta} \nonumber\\
&= f({\bm x}) + \mathcal{G}({\bm x})^\top{\bm \eta} + \frac{1}{2}{\bm \eta}^\top \nabla^2_{\mathbf{xx}}\mc{L}({\bm x}, \boldsymbol\lambda^\star({\bm x})){\bm \eta} \nonumber\\
&+ \frac{1}{2}{\bm \eta}^\top (\nabla^2_{\mathbf{xx}}\mathcal{L}({\bm x} + t{\bm \eta}, \boldsymbol\lambda^\star({{\bm x}})) - \nabla^2_{\mathbf{xx}}\mc{L}({\bm x}, \boldsymbol\lambda^\star({\bm x}))){\bm \eta} \label{feasible_expansion_estimate}.
\end{align}
Then the last term in $(\ref{feasible_expansion_estimate})$ can be bounded as follows
\begin{align*}
&\quad\quad  \vert \frac{1}{2}{\bm \eta}^\top \{\nabla^2_{\bm x \bm x}\mathcal{L}({\bm x} + t{\bm \eta}, \boldsymbol\lambda^\star({\bm x})) -
\nabla^2_{\bm x \bm x}\mathcal{L}({\bm x}, \boldsymbol\lambda^\star({\bm x}))\}{\bm \eta} \vert \\
&\leq \frac{1}{2}\|{\bm \eta}\|^2 \|\nabla^2_{\bm x \bm x}\mathcal{L}({\bm x} + t{\bm \eta}, \boldsymbol\lambda^\star({\bm x})) - \nabla^2_{\bm x \bm x}\mathcal{L}({\bm x}, \boldsymbol\lambda^\star({\bm x}))\|\\
&=\frac{1}{2}\|{\bm \eta}\|^2 \|(\nabla^2f({\bm x} + t{\bm \eta}) - \nabla^2f({\bm x})) + \sum_{i = 1}^m\lambda^\star_i({\bm x}) (\nabla^2c_i({\bm x} + t{\bm \eta}) - \nabla^2c_i({\bm x}))\| \\
&\leq \frac{1}{2}\|{\bm \eta}\|^2 \big( \|\nabla^2f({\bm x} + t{\bm \eta}) - \nabla^2f({\bm x})\| + \sum_{i=1}^m\vert\lambda^\star_i({{\bm x}})\vert \ \|\nabla^2c_i({\bm x} + t{\bm \eta}) - \nabla^2c_i({\bm x})\|  \big) \\
&\leq \frac{1}{2}\|{\bm \eta}\|^3 L_{f, 2}t + \frac{1}{2}\|{\bm \eta}\|^2\|\boldsymbol\lambda^\star({\bm x})\|\sqrt{\sum_{i=1}^m\|\nabla^2 c_i({\bm x} + t{\bm \eta}) - \nabla^2c_i({\bm x})\|^2} \tag{Cauchy-Schwartz inequailty}\\
&\leq \frac{1}{2}\|{\bm \eta}\|^3 L_{f, 2} + \frac{\gamma_{f,1} \sqrt{\Gamma_1}}{2\sigma_0^2}\sqrt{\sum_{i=1}^m L_{c_i, 2}^2t^2}\cdot \|{\bm \eta}\|^3 \tag{by (\ref{lambda_bound})}\\
&\leq \big\{\frac{L_{f, 2}}{2} + \frac{\gamma_{f, 1}}{2\sigma_0^2}\sqrt{\Gamma_1\Lambda_2}\big\}\|\bm\eta \|^3\\
&= C_1\|{\bm \eta}\|^3,
\end{align*}
where $C_1 = \big\{\frac{L_{f, 2}}{2} + \frac{\gamma_{f, 1}}{2\sigma_0^2}\sqrt{\Gamma_1\Lambda_2}\big\}$
Combining the above inequality with (\ref{feasible_expansion_estimate}), we obtain
\begin{align}\label{eqn:standard_taylor}
\big \vert f({\bm y}) -  f({\bm x}) - \mathcal{G}({\bm x})^\top \bm\eta - \frac{1}{2}{\bm \eta}^\top\{ \nabla^2_{\bm x \bm x}\mathcal{L}({\bm x}, \boldsymbol\lambda^\star({\bm x}))\}{\bm \eta}\big \vert \leq  C_1\|{\bm \eta}\|^3.
\end{align}
We consider two case: (i) $\bm{x}_0 + \boldsymbol\delta \in \Omega$, and (ii) $\bm{x}_0 + \boldsymbol\delta \not\in \Omega$.\\
Case (i): Substituting $\bm y = \bm x_0 + \bm \delta = \Pi_{\Omega}(\bm x_0 + \bm \delta)$ and $\bm x = \bm x_0$ into \eqref{eqn:standard_taylor}, and noting that $\boldsymbol\eta = \bm y - \bm x = \boldsymbol\delta$, we have from (\ref{eqn:standard_taylor}) and the facts that $\mc{H}(\bm{x}_0) = \mc{P}_{\bm{x}_0}^\top \nabla_{\bm x \bm x}^2 \mc{L}(\bm x_0, \bm \lambda^\star(\bm x_0)) \mc{P}_{\bm{x}_0}$ and $\bm \delta \in \mc{T}_\Omega(\bm x_0)$, hence that $\mc{P}_{\bm x_0}\bm\delta = \bm\delta$.
\begin{align*}
&\quad \big \vert f(\Pi_{\Omega}(\bm x_0 + \bm \delta)) -  f(\bm x_0) - \mathcal{G}({\bm x_0})^\top \bm\delta - \frac{1}{2}{\bm \delta}^\top\mc{H}(\bm x_0){\bm \delta}\big \vert \leq  C_1\|{\bm \delta}\|^3.
\end{align*}
Case (ii): ${\bm x}_0 + {\bm \delta} \not\in \Omega$. Letting ${\bm y_0} = \Pi_\Omega({\bm x}_0 +{\bm \delta})$, we have from (\ref{diff_bound}) and then from (\ref{eqn:standard_taylor}) that
\begin{flalign}\label{eqn:inequality_1}
&\quad\big\vert f({\bm y_0}) - f({\bm x}_0) - \mc{G}({\bm x}_0)^\top ({\bm y_0} - {\bm x}_0) - \frac{1}{2}({\bm y_0} - {\bm x}_0)^\top \{ \nabla^2_{\bm x \bm x}\mathcal{L}({\bm x_0}, \boldsymbol\lambda^\star({\bm x_0}))\}({\bm y_0} - {\bm x}_0) \big\vert \leq 8 C_1 \norm{\bm \delta}{}^3
\end{flalign}
Let $\bm \xi = (\bm x_0 + \bm \delta) - \bm y_0 = (\bm x_0 + \bm \delta) -  \Pi_\Omega({\bm x}_0 +{\bm \delta})$. Then clearly $\bm y_0 - \bm x_0 = \bm \delta - \bm \xi$, and \eqref{eqn:inequality_1} can be rewritten as
\begin{flalign} \label{expansion}
\big\vert f({\bm y_0}) - f({\bm x}_0) - \mc{G}({\bm x}_0)^\top (\bm \delta - \bm \xi) - \frac{1}{2}(\bm \delta - \bm \xi)^\top \{ \nabla^2_{\bm x \bm x}\mathcal{L}({\bm x_0}, \boldsymbol\lambda^\star({\bm x_0}))\}(\bm \delta - \bm \xi) \big\vert  &\le 8 C_1 \norm{\bm \delta}{}^3.
\end{flalign}
We further note that
\begin{align}
&\quad f({\bm y_0}) - f({\bm x}_0) - \mc{G}({\bm x}_0)^\top (\bm \delta - \bm \xi) - \frac{1}{2}(\bm \delta - \bm \xi)^\top \{ \nabla^2_{\bm x \bm x}\mathcal{L}({\bm x_0}, \boldsymbol\lambda^\star({\bm x_0}))\}(\bm \delta - \bm \xi)\\
&= f({\bm y_0}) - f({\bm x}_0) - \mc{G}(\bm x_0)^\top\bm \delta - \frac{1}{2}\bm \delta^\top \{ \nabla^2_{\bm x \bm x}\mathcal{L}({\bm x_0}, \boldsymbol\lambda^\star({\bm x_0}))\}\bm \delta \nonumber \\
&+ \mc{G}(\bm x_0)^\top\bm \xi + \bm \delta^\top \{ \nabla^2_{\bm x \bm x}\mathcal{L}({\bm x_0}, \boldsymbol\lambda^\star({\bm x_0}))\}\bm \xi - \frac{1}{2}\bm \xi^\top \{ \nabla^2_{\bm x \bm x}\mathcal{L}({\bm x_0}, \boldsymbol\lambda^\star({\bm x_0}))\}\bm \xi \label{eqn:last_three_terms}
\end{align}
Next, we will show that the last three terms in \eqref{eqn:last_three_terms} satisfy
\begin{flalign}
\mc{G}(\bm x_0)^\top\bm \xi + \bm \delta^\top \{ \nabla^2_{\bm x \bm x}\mathcal{L}({\bm x_0}, \boldsymbol\lambda^\star({\bm x_0}))\}\bm \xi - \frac{1}{2}\bm \xi^\top \{ \nabla^2_{\bm x \bm x}\mathcal{L}({\bm x_0}, \boldsymbol\lambda^\star({\bm x_0}))\}\bm \xi = O(\norm{\bm \delta}{}^3).
\end{flalign}
The following are helpful in establishing this. First, from the definition of $\bm \xi = (\bm x_0 + \bm \delta) -  \Pi_\Omega({\bm x}_0 +{\bm \delta})$ and Proposition 1, one has $\|{\bm \xi}\| \leq  4\|{\bm \delta}\|^2/R$ as well as the bound $\|\bm\xi\| \leq \|\bm \delta\|$. Second, since $\mc{G}({\bm x}_0)\in \mc{T}_\Omega(\bm x_0)$, we have $\mc{G}({\bm x}_0)^\top{\bm \xi} = \mc{G}({\bm x}_0)^\top\mc{P}_{{\bm x}_0}{\bm \xi}$. Third, since $\bm y_0 = \Pi_\Omega(\bm x_0 + \bm \delta)$, $\mc{P}_{\bm y_0}\{(\bm x_0 + \bm\delta) - \Pi_\Omega(\bm x_0 + \bm \delta)\} = \bm 0$. Using these observations and the fact that $\|{\bm y_0} - {\bm x}_0\| \leq 2\|{\bm \delta}\|$, we are ready to bound these three terms one by one,
\begin{align}
\vert {\bm \delta}^\top \{ \nabla^2_{\bm x \bm x}\mathcal{L}({\bm x_0}, \boldsymbol\lambda^\star({\bm x_0}))\}{\bm \xi} \vert  &\leq \{ \| \nabla^2f({\bm x}_0)\| + \sum_{i=1}^m \vert \lambda_i^*({\bm x}_0)\vert \ \|\nabla^2 c_i({\bm x}_0)\|\} \|{\bm \delta}\|\|{\bm \xi}\| \nonumber \\
&\leq  \Big\{\|\nabla^2f({\bm x}_0)\| + \|\boldsymbol\lambda^*({\bm x}_0)\|\sqrt{\sum_{i=1}^m  \|\nabla^2 c_i({\bm x}_0)\|^2}\Big\}  \|{\bm \delta}\|\|{\bm \xi}\| \tag{by Cauchy-Schwartz} \nonumber\\
&\leq \{\gamma_{f, 2} + \frac{\gamma_{f, 1}}{2\sigma_0^2}\sqrt{\Gamma_1\Gamma_2}\}\frac{4}{R}\|{\bm \delta}\|^3 \tag{by (\ref{lambda_bound}) and $\norm{\bm\xi}{}\leq 4\norm{\bm\delta}{}^2/R$} \nonumber\\
&= C_2 \|{\bm \delta}\|^3, \label{part_1}
\end{align}
where $C_2 = \{\gamma_{f, 2} + \frac{\gamma_{f, 1}}{2\sigma_0^2}\sqrt{\Gamma_1\Gamma_2}\}\frac{4}{R}$.
 Applying exactly the same method to bound $\frac{1}{2}{\bm \xi}^\top\{ \nabla^2_{\bm x \bm x}\mathcal{L}({\bm x_0}, \boldsymbol\lambda^\star({\bm x_0}))\}{\bm \xi}$ we have
\begin{align}
\frac{1}{2}\vert {\bm \xi}^\top\{ \nabla^2_{\bm x \bm x}\mathcal{L}({\bm x_0}, \boldsymbol\lambda^\star({\bm x_0}))\}{\bm \xi}\vert &\leq \frac{1}{2}\{ \| \nabla^2f({\bm x}_0)\| + \sum_{i=1}^m \vert \lambda_i^*({\bm x}_0)\vert \ \|\nabla^2 c_i({\bm x}_0)\|\}\|{\bm \xi}\|^2 \nonumber\\
 &\leq \{\gamma_{f, 2} + \frac{\gamma_{f, 1}}{2\sigma_0^2}\sqrt{\Gamma_1\Gamma_2}\}\frac{1}{2}\|{\bm \xi}\|^2  \tag{by (\ref{lambda_bound})}\\
 &\leq \{\gamma_{f, 2} + \frac{\gamma_{f, 1}}{2\sigma_0^2}\sqrt{\Gamma_1\Gamma_2}\}\frac{2}{R}\|{\bm \delta}\|^3 \tag{ by $\norm{\bm\xi}{}\leq 4\norm{\bm\delta}{}^2/R$ and $\|\bm\xi\| \leq \|\bm\delta\|$} \nonumber \\
&= C_3 \|{\bm \delta}\|^3, \label{part_2}
\end{align}
where $C_3 = \{\gamma_{f, 2} + \frac{\gamma_{f, 1}}{2\sigma_0^2}\sqrt{\Gamma_1\Gamma_2}\}\frac{2}{R}$.
To bound the last term $\mc{G}(\bm{x}_0)^\top\bm\xi$, we first note that $\mc{G}(\bm{x}_0)^\top\bm\xi = \mc{G}(\bm{x}_0)^\top\mc{P}_{\bm{x}_0}\bm\xi$ since $\mc{G}(\bm{x}_0) \in \mc{T}_{\Omega}(\bm{x}_0)$ and $\mc{P}_{\bm{y}_0}\bm\xi = \bm 0$ since $\bm\xi = \bm{y}_0 - \Pi_\Omega \bm{y}_0$. Therefore
\begin{align}
\mc{G}(\bm{x}_0)^\top \bm\xi = \mc{G}(\bm{x}_0)^\top \mc{P}_{\bm{x}_0}\bm\xi = \mc{G}(\bm{x}_0)^\top\{\mc{P}_{\bm{x}_0}\bm\xi - \mc{P}_{\bm{y}_0}\bm\xi\} \leq \norm{\mc{G}(\bm{x}_0)}{}\norm{\mc{P}_{\bm{x}_0}\bm\xi - \mc{P}_{\bm{y}_0}\bm\xi}{}. \label{tough_bound}
\end{align}
We first derive a bound for $\norm{\mc{P}_{\bm{x}_0}\bm\xi - \mc{P}_{\bm{y}_0}\bm\xi}{}$. Note that $\mc{P}_{\bm{x}} = \nabla\bm{c}({\bm x})\{\nabla\bm{c}({\bm x})^\top \nabla\bm{c}({\bm x})\}^{-1}\nabla\bm{c}({\bm x})^\top$; thus
\begin{align}
\norm{\mc{P}_{\bm{x}_0}\bm\xi - \mc{P}_{\bm{y}_0}\bm\xi}{} &\leq \norm{\mc{P}_{\bm{x}_0} - \mc{P}_{\bm{y}_0}}{}\norm{\bm\xi}{} \nonumber \\
&= \norm{\bm\xi}{}\norm{\nabla\bm{c}({\bm x_0})\{\nabla\bm{c}({\bm x_0})^\top \nabla\bm{c}({\bm x_0})\}^{-1}\nabla\bm{c}({\bm x_0})^\top -\nabla\bm{c}({\bm y_0})\{\nabla\bm{c}({\bm y_0})^\top \nabla\bm{c}({\bm y_0})\}^{-1}\nabla\bm{c}({\bm y_0})^\top }{}. \nonumber \\
&= \norm{\bm\xi}{} \|[\nabla\bm{c}({\bm x_0})\{\nabla\bm{c}({\bm x_0})^\top \nabla\bm{c}({\bm x_0})\}^{-1}\nabla\bm{c}({\bm x_0})^\top - \nabla\bm{c}({\bm x_0})\{\nabla\bm{c}({\bm x_0})^\top \nabla\bm{c}({\bm x_0})\}^{-1}\nabla\bm{c}({\bm y_0})^\top] \nonumber \\
&\quad + [\nabla\bm{c}({\bm x_0})\{\nabla\bm{c}({\bm x_0})^\top \nabla\bm{c}({\bm x_0})\}^{-1}\nabla\bm{c}({\bm y_0})^\top - \nabla\bm{c}({\bm y_0})\{\nabla\bm{c}({\bm y_0})^\top \nabla\bm{c}({\bm y_0})\}^{-1}\nabla\bm{c}({\bm y_0})^\top] \| \tag{adding and subtracting a term} \nonumber \\
&\leq \norm{\bm\xi}{} \Big\{ \|\nabla\bm{c}({\bm x_0})\{\nabla\bm{c}({\bm x_0})^\top \nabla\bm{c}({\bm x_0})\}^{-1}\nabla\bm{c}({\bm x_0})^\top - \nabla\bm{c}({\bm x_0})\{\nabla\bm{c}({\bm x_0})^\top \nabla\bm{c}({\bm x_0})\}^{-1}\nabla\bm{c}({\bm y_0})^\top \| \nonumber \\
&\quad +\| \nabla\bm{c}({\bm x_0})\{\nabla\bm{c}({\bm x_0})^\top \nabla\bm{c}({\bm x_0})\}^{-1}\nabla\bm{c}({\bm y_0})^\top - \nabla\bm{c}({\bm y_0})\{\nabla\bm{c}({\bm y_0})^\top \nabla\bm{c}({\bm y_0})\}^{-1}\nabla\bm{c}({\bm y_0})^\top \| \Big\} \tag{triangle inequality} \\
&\leq \norm{\bm\xi}{} \Big\{ \|\nabla\bm{c}({\bm x_0})\{\nabla\bm{c}({\bm x_0})^\top \nabla\bm{c}({\bm x_0})\}^{-1} \| \| \nabla\bm{c}({\bm x_0}) - \nabla\bm{c}({\bm y_0})\| \label{norm_bound_1}\\
&\quad + \|  \nabla\bm{c}({\bm x_0})\{\nabla\bm{c}({\bm x_0})^\top \nabla\bm{c}({\bm x_0})\}^{-1} -  \nabla\bm{c}({\bm y_0})\{\nabla\bm{c}({\bm y_0})^\top \nabla\bm{c}({\bm y_0})\}^{-1} \| \| \nabla\bm{c}({\bm y_0}) \| \Big\} \label{norm_bound_2}.
\end{align}
Since
\begin{align}
\norm{\nabla \bm c(\bm x_0) - \nabla \bm c(\bm y_0)}{} \leq \norm{\nabla \bm c(\bm x_0) - \nabla \bm c(\bm y_0)}{F} &= \sqrt{\sum_{i=1}^m \vert \nabla c_i(\bm x_0) - \nabla c_i(\bm y_0) \vert^2} \nonumber \\
&\leq \sqrt{\sum_{i=1}^m L_{c_i, 1}^2} \norm{\bm x_0 - \bm y_0}{} = \sqrt\Lambda_1\norm{\bm x_0 - \bm y_0}{}, \label{difference_bound}
\end{align}
and
\begin{align}
\norm{\nabla \bm c(\bm y_0)}{} \leq \norm{\nabla \bm c(\bm y_0)}{F} = \sqrt{\sum_{i=1}^m \|\nabla c_i(\bm y_0) \|^2} \leq \sqrt{\sum_{i=1}^m \gamma_{c_i, 1}^2} = \sqrt{\Gamma_1}, \label{single_bound}
\end{align}
Moreover, by
\begin{align*}
\|\nabla\bm{c}({\bm x_0})\{\nabla\bm{c}({\bm x_0})^\top \nabla\bm{c}({\bm x_0})\}^{-1} \| &\leq \| \nabla\bm{c}({\bm x_0}) \| \|\{\nabla\bm{c}({\bm x_0})^\top \nabla\bm{c}({\bm x_0})\}^{-1} \| \\
&\leq \frac{\sqrt{\Gamma_1}}{\sigma_0^2}. \tag{by (\ref{single_bound}) and $\sigma_0$-LICQ condition}
\end{align*}
Therefore $\|\nabla\bm{c}({\bm x_0})\{\nabla\bm{c}({\bm x_0})^\top \nabla\bm{c}({\bm x_0})\}^{-1} \|\|\nabla \bm c(\bm x_0) - \nabla \bm c(\bm y_0)\| $ in (\ref{norm_bound_1}) can be bounded by
\begin{align}
\|\nabla\bm{c}({\bm x_0})\{\nabla\bm{c}({\bm x_0})^\top \nabla\bm{c}({\bm x_0})\}^{-1} \|\|\nabla \bm c(\bm x_0) - \nabla \bm c(\bm y_0)\| \leq \frac{\sqrt{\Gamma_1\Lambda_1}}{\sigma_0^2}\| \bm x_0 - \bm y_0 \|. \label{result_part_1}
\end{align}

  We need to further simplify $\|  \nabla\bm{c}({\bm x_0})\{\nabla\bm{c}({\bm x_0})^\top \nabla\bm{c}({\bm x_0})\}^{-1} -  \nabla\bm{c}({\bm y_0})\{\nabla\bm{c}({\bm y_0})^\top \nabla\bm{c}({\bm y_0})\}^{-1} \|$ in order to obtain an upper bound for (\ref{norm_bound_1}).
\begin{align}
&\quad \|  \nabla\bm{c}({\bm x_0})\{\nabla\bm{c}({\bm x_0})^\top \nabla\bm{c}({\bm x_0})\}^{-1} -  \nabla\bm{c}({\bm y_0})\{\nabla\bm{c}({\bm y_0})^\top \nabla\bm{c}({\bm y_0})\}^{-1} \| \nonumber \\
&= \| [\nabla\bm{c}({\bm x_0})\{\nabla\bm{c}({\bm x_0})^\top \nabla\bm{c}({\bm x_0})\}^{-1} - \nabla\bm{c}({\bm y_0})\{\nabla\bm{c}({\bm x_0})^\top \nabla\bm{c}({\bm x_0})\}^{-1}] \nonumber \\
&\quad + [\nabla\bm{c}({\bm y_0})\{\nabla\bm{c}({\bm x_0})^\top \nabla\bm{c}({\bm x_0})\}^{-1} -  \nabla\bm{c}({\bm y_0})\{\nabla\bm{c}({\bm y_0})^\top \nabla\bm{c}({\bm y_0})\}^{-1}] \| \nonumber\tag{add and subtract terms} \\
&\leq \| \nabla\bm{c}({\bm x_0})\{\nabla\bm{c}({\bm x_0})^\top \nabla\bm{c}({\bm x_0})\}^{-1} - \nabla\bm{c}({\bm y_0})\{\nabla\bm{c}({\bm x_0})^\top \nabla\bm{c}({\bm x_0})\}^{-1} \| \nonumber \\
&\quad + \| \nabla\bm{c}({\bm y_0})\{\nabla\bm{c}({\bm x_0})^\top \nabla\bm{c}({\bm x_0})\}^{-1} -  \nabla\bm{c}({\bm y_0})\{\nabla\bm{c}({\bm y_0})^\top \nabla\bm{c}({\bm y_0})\}^{-1}] \| \nonumber \tag{triangle inequality} \\
&\leq \| \nabla \bm c(\bm x_0) - \nabla \bm c(\bm y_0) \| \{\nabla\bm{c}({\bm x_0})^\top \nabla\bm{c}({\bm x_0})\}^{-1} \| \label{sub_bound_1}\\
&\quad + \| \nabla \bm c(\bm y_0) \| \|\{\nabla\bm{c}({\bm x_0})^\top \nabla\bm{c}({\bm x_0})\}^{-1} - \{\nabla\bm{c}({\bm y_0})^\top \nabla\bm{c}({\bm y_0})\}^{-1} \| \label{sub_bound_2}.
\end{align}
An upper bound of (\ref{sub_bound_1}) can be obtained by combining (\ref{difference_bound}) and the $\sigma_0$-LICQ condition, that is,
\begin{align}
\| \nabla \bm c(\bm x_0) - \nabla \bm c(\bm y_0) \| \{\nabla\bm{c}({\bm x_0})^\top \nabla\bm{c}({\bm x_0})\}^{-1} \| \leq \frac{\sqrt{\Lambda_1}}{\sigma^2}\|\bm x_0 - \bm y_0 \|. \label{bound_for_sub_bound_1}
\end{align}
To upper bound $ \|\{\nabla\bm{c}({\bm x_0})^\top \nabla\bm{c}({\bm x_0})\}^{-1} - \{\nabla\bm{c}({\bm y_0})^\top \nabla\bm{c}({\bm y_0})\}^{-1} \|$ in (\ref{sub_bound_2}), we need to utilize the fact that for any invertible matrices $A, B$, $\norm{A^{-1} - B^{-1}}{} \leq \norm{A^{-1}}{}\norm{A - B}{}\norm{B^{-1}}{}$ and the $\sigma_0$-LICQ condition. More specifically,
\begin{align}
 &\quad \|\{\nabla\bm{c}({\bm x_0})^\top \nabla\bm{c}({\bm x_0})\}^{-1} - \{\nabla\bm{c}({\bm y_0})^\top \nabla\bm{c}({\bm y_0})\}^{-1} \| \nonumber \\
 &\leq \|\{\nabla\bm{c}({\bm x_0})^\top \nabla\bm{c}({\bm x_0})\}^{-1} \| \| \{\nabla\bm{c}({\bm y_0})^\top \nabla\bm{c}({\bm y_0})\}^{-1} \| \|\nabla\bm{c}({\bm x_0})^\top \nabla\bm{c}({\bm x_0}) - \nabla\bm{c}({\bm y_0})^\top \nabla\bm{c}({\bm y_0}) \|\nonumber  \\
 &\leq \frac{1}{\sigma_0^4}\|\nabla\bm{c}({\bm x_0})^\top \nabla\bm{c}({\bm x_0}) - \nabla\bm{c}({\bm y_0})^\top \nabla\bm{c}({\bm y_0}) \| \tag{$\sigma_0$-LICQ condition} \nonumber \\
 &= \frac{1}{\sigma_0^4} \|\nabla\bm{c}({\bm x_0})^\top \nabla\bm{c}({\bm x_0}) - \nabla\bm{c}({\bm x_0})^\top \nabla\bm{c}({\bm y_0}) + \nabla\bm{c}({\bm x_0})^\top \nabla\bm{c}({\bm y_0})  - \nabla\bm{c}({\bm y_0})^\top \nabla\bm{c}({\bm y_0}) \| \nonumber \\
 &\leq \frac{1}{\sigma_0^4}\| \nabla \bm c(\bm x_0) - \nabla \bm c(\bm y_0) \|\big\{\| \nabla \bm c(\bm x_0) \| + \| \nabla \bm c(\bm{y}_0) \| \} \nonumber \\
 &\leq \frac{2\sqrt{\Gamma_1\Lambda_1}}{\sigma_0^4}\|\bm x_0 - \bm y_0 \|. \label{inverse_bound}
\end{align}
Therefore we can bound $\|  \nabla\bm{c}({\bm x_0})\{\nabla\bm{c}({\bm x_0})^\top \nabla\bm{c}({\bm x_0})\}^{-1} -  \nabla\bm{c}({\bm y_0})\{\nabla\bm{c}({\bm y_0})^\top \nabla\bm{c}({\bm y_0})\}^{-1} \|$ in (\ref{norm_bound_2}) by
\begin{align}
&\quad \|  \nabla\bm{c}({\bm x_0})\{\nabla\bm{c}({\bm x_0})^\top \nabla\bm{c}({\bm x_0})\}^{-1} -  \nabla\bm{c}({\bm y_0})\{\nabla\bm{c}({\bm y_0})^\top \nabla\bm{c}({\bm y_0})\}^{-1} \| \nonumber \\
&\leq \| \nabla \bm c(\bm x_0) - \nabla \bm c(\bm y_0) \| \{\nabla\bm{c}({\bm x_0})^\top \nabla\bm{c}({\bm x_0})\}^{-1} \| \nonumber \\
&\quad + \| \nabla \bm c(\bm y_0) \| \|\{\nabla\bm{c}({\bm x_0})^\top \nabla\bm{c}({\bm x_0})\}^{-1} - \{\nabla\bm{c}({\bm y_0})^\top \nabla\bm{c}({\bm y_0})\}^{-1} \| \nonumber \\
&\leq(\frac{\sqrt{\Lambda_1}}{\sigma^2_0} + \frac{2\Gamma_1\sqrt{\Lambda_1}}{\sigma_0^4}) \|\bm x_0 - \bm y_0 \| \label{bound_for_norm_bound_2}.
\end{align}
With the upper bounds for (\ref{norm_bound_1}) and (\ref{norm_bound_2}), we bound $\norm{\mc{P}_{\bm x_0}\bm\xi - \mc{P}_{\bm y_0}\bm\xi}{}$ by
\begin{align*}
&\quad \norm{\mc{P}_{\bm x_0}\bm\xi - \mc{P}_{\bm y_0}\bm\xi}{} \\
&\leq \norm{\bm\xi}{} \Big\{ \|\nabla\bm{c}({\bm x_0})\{\nabla\bm{c}({\bm x_0})^\top \nabla\bm{c}({\bm x_0})\}^{-1} \| \| \nabla\bm{c}({\bm x_0}) - \nabla\bm{c}({\bm y_0})\| \\
&\quad + \|  \nabla\bm{c}({\bm x_0})\{\nabla\bm{c}({\bm x_0})^\top \nabla\bm{c}({\bm x_0})\}^{-1} -  \nabla\bm{c}({\bm y_0})\{\nabla\bm{c}({\bm y_0})^\top \nabla\bm{c}({\bm y_0})\}^{-1} \| \| \nabla\bm{c}({\bm y_0}) \| \Big\} \\
&\leq \{\frac{\sqrt{\Gamma_1\Lambda_1}}{\sigma_0^2}\| \bm x_0 - \bm y_0 \| + \sqrt{\Gamma_1}(\frac{\sqrt{\Lambda_1}}{\sigma^2_0} + \frac{2\Gamma_1\sqrt{\Lambda_1}}{\sigma_0^4}) \|\bm x_0 - \bm y_0 \|\} \|\bm\xi\| \tag{by (\ref{result_part_1}), (\ref{single_bound}), and (\ref{bound_for_norm_bound_2})} \\
&= (\frac{2\sqrt{\Gamma_1\Lambda_1}}{\sigma_0^2} + \frac{2\sqrt{\Gamma_1^3\Lambda_1}}{\sigma_0^4})\|\bm x_0 - \bm y_0 \|\|\bm\xi\| \\
&\leq (\frac{2\sqrt{\Gamma_1\Lambda_1}}{\sigma_0^2} + \frac{2\sqrt{\Gamma_1^3\Lambda_1}}{\sigma_0^4}) \frac{8\|\bm\delta\|^3}{R}. \tag{$\|\bm x_0 - \bm y_0\|\leq 2 \| \bm \delta \|$ and $\| \bm \xi \| \leq 4\|\bm \delta \|^2 / R$}
\end{align*}
Let us go back to the task of bounding $\mc{G}(\bm x_0)^\top \bm\xi$. By (\ref{tough_bound}) and the inequality above, we have
\begin{align}
\mc{G}(\bm x_0)^\top \bm\xi &\leq \norm{\mc{G}(\bm{x}_0)}{}\norm{\mc{P}_{\bm{x}_0}\bm\xi - \mc{P}_{\bm{y}_0}\bm\xi}{} \nonumber\\
 &\leq \norm{\mc{G}(\bm{x}_0)}{}(\frac{2\sqrt{\Gamma_1\Lambda_1}}{\sigma_0^2} + \frac{2\sqrt{\Gamma_1^3\Lambda_1}}{\sigma_0^4}) \frac{8\|\bm\delta\|^3}{R} \nonumber \\
 &\leq \Big\{\|\nabla f({\bm x}_0)\| + \|\boldsymbol\lambda^*(\bm x_0)\|\sqrt{\sum_{i=1}^m\|\nabla c_i({\bm x}_0)\|^2}\Big\}(\frac{2\sqrt{\Gamma_1\Lambda_1}}{\sigma_0^2} + \frac{2\sqrt{\Gamma_1^3\Lambda_1}}{\sigma_0^4}) \frac{8\|\bm\delta\|^3}{R} \tag{by the definition of $\mc{G}(\bm x_0)$ and triangle inequality} \nonumber \\
 &\leq \{\gamma_{f, 1} + \frac{\gamma_{f, 1}}{\sigma_0^2}\Gamma_1\}(\frac{2\sqrt{\Gamma_1\Lambda_1}}{\sigma_0^2} + \frac{2\sqrt{\Gamma_1^3\Lambda_1}}{\sigma_0^4}) \frac{8\|\bm\delta\|^3}{R}. \tag{by (\ref{lambda_bound}) and Assumption 2} \nonumber \\
 &= C_4 \|\bm\delta\|^3, \label{part_3}
\end{align}
where $C_4 = \{\gamma_{f, 1} + \frac{\gamma_{f, 1}}{\sigma_0^2}\Gamma_1\}(\frac{2\sqrt{\Gamma_1\Lambda_1}}{\sigma_0^2} + \frac{2\sqrt{\Gamma_1^3\Lambda_1}}{\sigma_0^4}) \frac{8}{R}$.
Finally, we can bound the last three terms in \eqref{eqn:last_three_terms} by combining (\ref{part_1}), (\ref{part_2}, and (\ref{part_3}), that is,
\begin{align}
\vert \mc{G}(\bm x_0)^\top\bm \xi + \bm \delta^\top \{ \nabla^2_{\bm x \bm x}\mathcal{L}({\bm x_0}, \boldsymbol\lambda^\star({\bm x_0}))\}\bm \xi - \frac{1}{2}\bm \xi^\top \{ \nabla^2_{\bm x \bm x}\mathcal{L}({\bm x_0}, \boldsymbol\lambda^\star({\bm x_0}))\}\bm \xi \vert \leq (C_2 + C_3 + C_4)\| \bm\delta \|^3. \label{bound_last_three_terms}
\end{align}
Plugging (\ref{bound_last_three_terms}) into (\ref{expansion}) yields the final result, we obtain,
\begin{align*}
&\quad \vert f(\bm y_0) - f(\bm x_0) - \mc{G}(\bm x_0)^\top\bm\delta-\frac{1}{2}\bm\delta^\top\{\nabla^2_{\bm{xx}}\mc{L}(\bm x_0, \bm\lambda^*(\bm x_0))\}\bm\delta \vert \\
&=\vert f(\bm y_0) - f(\bm x_0) - \mc{G}(\bm x_0)^\top\bm\delta-\frac{1}{2}\bm\delta^\top\mc{H}(\bm x_0)\bm\delta \vert \tag{by $\mc{P}_{\bm x_0}\bm\delta = \bm\delta$ and definition of $\mc{H}(\bm{x}_0)$} \\
&\leq (8C_1 + C_2 + C_3 + C_4)\| \bm\delta \|^3.
\end{align*}

\end{proof}
\subsection{Riemannian gradient and Riemannian Hessian}\label{sub:relations}
Recall that  $\Omega = \{ {\bm x} \in \mr{R}^n \; \vert \; c_i({\bm x}) = 0, i= 1, \ldots, m\}$, $\nabla\bm{c}({\bm x}) = [\nabla c_1({\bm x}), \cdots, \nabla c_m({\bm x})]$ and $\sigma_{\min}({\nabla\bm c}({\bm x}))$ the minimum singular value of ${\nabla \bm c}({\bm x})$. Assume that $\inf\{\sigma_{\min}({\nabla \bm c}({\bm x})) \ \vert \ {\bm x} \in \Omega\} > \alpha$ for some $\alpha > 0$. For second-order differentiable functions $f$ and $c_i(\bm{x}), i = 1, \ldots, m$, let $\mc{L}({\bm x},\boldsymbol\lambda) = f({\bm x}) - \sum_{i=1}^m \lambda_i c_i({\bm x})$ and $\boldsymbol \lambda^\star({\bm x}) = \arg\!\min_{\boldsymbol\lambda}\|\nabla_{\bm x} \mc{L}({\bm x}, \boldsymbol\lambda)\|$. Define $\mc{G}({\bm x}) = \nabla_{\bm x} \mc{L}({\bm x}, \boldsymbol\lambda^\star)$ and $\mc{H}({\bm x}) = \nabla^2_{{\bm x\bm x}} \mc{L}({\bm x}, \boldsymbol\lambda^\star)$. Note that $\Omega$ can be considered as a $n-m$ dimensional Riemannian sub-manifold of $\mr{R}^n$ under the LICQ assumption. We will show that $\mc{G}(\cdot)$ and $\mc{H}(\cdot)$  are Riemannian gradient and Riemannian Hessian of $f$ over $\Omega$ in the following lemma. A similar argument can be found in \cite{absil2013extrinsic} and \cite{absil2009all}
\begin{lemma}
$\mc{G}(\cdot)$ and $\mc{H}(\cdot)$ defined in \ref{eqn:G} and \ref{eqn:H} are the Riemannian gradient and Riemannian Hessian of $f(\cdot)$ over $\Omega$.
\end{lemma}
\begin{proof}
Let $\text{grad} f({\bm x})$ denote the Riemannian gradient of $f$ at ${\bm x} \in \Omega$. By definition, for every $\eta \in \mc{T}_\Omega({{\bm x}})$, $Df({\bm x})[\eta] = \langle \text{grad} f({\bm x}), \eta \rangle_{\bm x}$ where $Df({\bm x})[\eta]$ denotes the directional derivative of $f$ at ${\bm x}$ along the direction $\eta$ and $\langle \cdot, \cdot \rangle_{\bm x}$ denotes the Riemannian metric on $\mc{T}_\Omega({{\bm x}})$. Since $\Omega$ is an embedded Riemannian sub-manifold of $\mr{R}^n$, $\langle \cdot, \cdot \rangle_{\bm x}$ coincides with the Euclidean inner product. Therefore $\mc{G}({\bm x})$ is the projection of $\nabla f({\bm x})$ onto $\mc{T}_{\Omega}({{\bm x}})$. Note that for every $\boldsymbol\eta \in \mc{T}_{\Omega}({{\bm x}})$, we have ${\nabla\bm c}({\bm x})\eta = \boldsymbol 0$. Then, we have
\begin{align*}
\langle \mc{G}({\bm x}), \boldsymbol\eta \rangle &= \langle \nabla f({\bm x}) - {\nabla\bm c}({\bm x})\boldsymbol\lambda^\star({\bm x}), \boldsymbol\eta \rangle = \langle \nabla f({\bm x}), \boldsymbol\eta \rangle - \boldsymbol\lambda^\star({\bm x})^\top{\nabla \bm c}({\bm x})^\top\boldsymbol\eta = \langle \nabla f({\bm x}), \boldsymbol\eta \rangle = \text{D}f({\bm x})[\boldsymbol\eta].
\end{align*}
The last equality follows from the definition of directional derivative. Therefore $\langle \text{grad} f({\bm x}), \boldsymbol\eta \rangle_{\bm x} = \langle \mc{G}({\bm x}), \boldsymbol\eta\rangle_{\bm x}$ for every $\boldsymbol\eta \in \mc{T}_\Omega({{\bm x}})$, that is, $\mc{G}({\bm x}) = \text{grad}f({\bm x})$. Next, we will show that $\mc{H}({\bm x})$ is the Riemannian Hessian of $f$.

  Let $\text{Hess}f({\bm x})$ denote the Riemannian Hessian of $f$ at ${\bm x} \in \Omega$. By definition, for all ${\bm \xi}, \boldsymbol\eta \in \mc{T}_\Omega({\bm x})$, $\text{Hess}f({\bm x})[{\bm \xi}, \boldsymbol\eta] = \langle \overline{\nabla}_\xi\text{grad}f({\bm x}), \boldsymbol\eta \rangle_{\bm x}$ where $\overline{\nabla}$ denotes the Riemannian connection on $\Omega$. First note that under the LICQ assumption, ${\nabla \bm c}({\bm x})^\top{\nabla\bm c}({\bm x})$ is invertible for every ${\bm x} \in \Omega$ and straight forward calculation gives $\boldsymbol\lambda^\star({\bm x}) = ({\nabla \bm c}({\bm x})^\top{\nabla\bm c}({\bm x}))^{-1}{\nabla\bm c}({\bm x})^\top\nabla f({\bm x})$. By the inverse function theorem, $\boldsymbol\lambda^\star({\bm x})$ is differentiable. By Proposition 5.3.2 in \cite{AMS09} and $\mc{G}({\bm x}) = \text{grad} f({\bm x})$,
\begin{align*}
\overline{\nabla}_\xi\text{grad}f({\bm x}) &= \mc{P}_{\bm{x}}(\text{D}\ \text{grad}f({\bm x})[\bm\xi]) \\
&= \mc{P}_{\bm{x}}(\nabla \mc{G}(\mb(x)) \bm\xi) \\
&= \mc{P}_{\bm{x}}\{\nabla^2 f({\bm x})\bm\xi -\sum_{i=1}^m (\nabla\lambda^\star_i({\bm x})\nabla c_i({\bm x})^\top\bm\xi + \lambda^\star_i\nabla^2c_i({\bm x})\bm\xi ) \} \\
&= \mc{P}_{\bm{x}}\{\nabla^2 f({\bm x})\bm\xi -\sum_{i=1}^m\lambda^\star_i\nabla^2c_i({\bm x})\bm\xi \} \\
&= \mc{P}_{\bm{x}}\nabla_{{\bm x\bm x}}^2 \mc{L}({\bm x}, \boldsymbol\lambda^\star({\bm x}))\bm\xi\\
&= \mc{P}_{\bm{x}}\mc{H}({\bm x})\mc{P}_{\bm{x}}\bm\xi.
\end{align*}	
Therefore $\text{Hess}f({\bm x})[{\bm \xi}, \boldsymbol\eta] = \mc{H}({\bm x})({\bm \xi}, \boldsymbol\eta)$.
\end{proof}

\section*{Acknowledgment}
This research was supported in part by NSF Grant CCF1527809.

\bibliographystyle{unsrt}
\bibliography{curvilinear}
\end{document}